\documentclass[11pt]{amsart}

\usepackage[all]{xy}
\usepackage{graphicx}
\usepackage{amsmath,amsxtra,amssymb,latexsym, amscd,amsthm,enumerate}
\usepackage{indentfirst}
\usepackage[mathscr]{eucal}
\usepackage{stackrel}
\usepackage{float} 
 
\newtheorem{thm}{Theorem}[section]
\newtheorem{cor}[thm]{Corollary}
\newtheorem{lem}[thm]{Lemma}
\newtheorem{prop}[thm]{Proposition}

\theoremstyle{definition}
\newtheorem{defn}[thm]{Definition}

\newtheorem{rem}[thm]{Remark}

\newcommand {\Coker} {{\rm Coker}}
\newcommand {\Hom} {{\rm Hom}}
\newcommand {\Ext} {{\rm Ext}}

\newcommand {\grade} {{\rm grade}}
\newcommand {\depth} {{\rm depth}}
\newcommand {\Ann} {{\rm Ann}}
\newcommand {\m} {{\mathfrak m}}
\newcommand {\p} {{\mathfrak p}}
\newcommand {\q} {{\mathfrak q}}
\newcommand {\Supp} {{\rm Supp}}
\newcommand {\codim} {{\rm codim}}
\newcommand {\dimproj} {{\rm projdim}}
\newcommand {\het} {{\rm ht}}

\newcommand {\N} {{\mathbb N}}

\newcommand {\0} {{\bf 0}}

\newcommand{\igno} [1] {} 
 
\newcommand {\Sing} {{\rm Sing}}
\newcommand {\Sp} {{\rm Spec}}
\newcommand {\K} {{\mathbb K}}
\newcommand {\Pol}{\K[x_1,\ldots,x_s]}

\newcommand{\Om}{\Omega^{1}_{R/K}}
\newcommand{\Omr}{\Omega^{1}_{R/\K}}

\newcommand{\Omn}{\Omega^{(n)}_{R/K}}
\newcommand{\Omk}{\Omega^{(n)}_{R/\K}}

\newcommand{\Oma}{\Omega^{(n)}_{A_{\p}/\K}}
\newcommand{\Omw}{\Omega^{(n)}_{W/\K}}
\newcommand {\Ap} {A_{\p}}

\begin{document}

\title[$k$-torsion of modules of high order differentials]{On the $k$-torsion of the module of differentials of order $n$ of hypersurfaces}
 
\author{Hern\'an de Alba}
\author{Daniel Duarte}
\email{H. de Alba, hdealbaca@conacyt.mx}
\email{D. Duarte (corresponding author), aduarte@uaz.edu.mx}
\address{Universidad Aut\'onoma de Zacatecas - CONACYT, Calzada Solidaridad y Paseo de la Bufa, Zacatecas, Zac. 98000, Mexico}
 
\thanks{H. de Alba was supported by CONACyT grant A1-S-30482}
\thanks{D. Duarte was supported by CONACyT grant 287622}
\date{\today}
\dedicatory{}
\commby{}

\begin{abstract} 
We characterize the $k$-torsion freeness of the module of differentials of order $n$ of a point of a hypersurface 
in terms of the singular locus of the corresponding local ring.
\end{abstract}

\maketitle



\section{Introduction}

The module of K\"ahler differentials of a ring is a classical object in commutative algebra. Recall that for a $K$-algebra $R$,
the module of K\"ahler differentials is defined as the quotient $\Omega^1_{R/K}:=I_R/I_R^2$, where $I_R$ is the kernel of the 
multiplication map $R\otimes_K R\rightarrow R$. More generally, the module of K\"ahler differentials of order $n$ can be defined as
$\Omega^{(n)}_{R/K}:=I_R/I_R^{n+1}$ (see, for instance, \cite{G,N,O}).

It is well-known that the module of differentials can be used to detect properties of the ring. For instance, under some hypothesis, 
the regularity of the localization of a finitely generated algebra is equivalent to the freeness of its module of differentials. An analogous 
statement holds for the module of high order differentials  (this was proved for hypersurfaces in \cite{BD} and, in a 
more general context, in \cite{BJNB}). 

We are interested in studying other properties of certain rings that can be detected through its module of differentials. Let $V$ be an
affine variety over a perfect field $\K$. Suppose that $V$ is locally, at some point $P\in V$, a complete intersection. Denote as $R$
the corresponding local ring. It was proved by J. Lipman that $V$ being non-singular at $P$ in codimension 1 (resp. in codimension 2) 
is equivalent to the torsion freeness (resp. reflexiveness) of $\Omr$ (see \cite{L}). It was proved that the 
first statement of Lipman's theorem also holds for the module of high order differentials in the case of hypersurfaces (see \cite{BD}).

There is a general notion of $k$-torsion freeness for any $k\in\N$, $k\geq1$, that generalizes the notions of torsion freeness and 
reflexiveness (see \cite{AB} or section \ref{sect k tors} below). The main goal of this paper is to prove that $k$-torsion freeness of 
the module of high order differentials of a hypersurface can be characterized in terms of the singular locus. 

Our approach to the problem is essentially the same as Lipman's. After making a careful analysis of his proof, we realized that part
of the arguments were valid in a much more general situation. In addition, a key ingredient in Lipman's proof is the fact that 
the projective dimension of the module of differentials of a reduced locally complete intersection is less or equal than one. An analogous 
statement was proved in \cite{BD} for the module of high order differentials of hypersurfaces, allowing us to carry on with
Lipman's strategy. Finally, the last ingredient we need for our proof is a criterion of regularity for hypersurfaces in terms of the module 
of high order differentials.


\section{Modules of K\"ahler differentials}

In this paper, all rings we consider are assumed to be commutative and with a unit element. 

Let $R$ be a $K$-algebra. Denote $I_R$ the kernel of the homomorphism $R\otimes_{K}R\rightarrow R$, $r\otimes s\mapsto rs$. 
Giving structure of $R$-module to $R\otimes_{K}R$ by multiplying on the left, define the $R$-module 
$$\Omn:=I_R/I_R^{n+1}.$$

\begin{defn}\cite[Definition 1.5]{O}
The $R$-module $\Omn$ is called the \textit{module of K\"ahler differentials of order n} of $R$ over $K$ or the \textit{module of high 
order K\"ahler differentials}. For $n=1$, this is just the usual module of K\"ahler differentials of $R$.
\end{defn}

A classical result states that, under some hypothesis, the localization of a finitely generated algebra $R$ is regular if and only if $\Om$ is free 
(see, for instance, \cite[Chapter II, Theorem 8.8]{H}). Another result in this direction is the following theorem due to J. Lipman (the statement 
(1) was also proved by S. Suzuki in \cite{S}).

\begin{thm}\cite[Proposition 8.1]{L}\label{thm L}
Let $R$ be the local ring of a point $P$ on an affine variety $V$ over a perfect field $\K$. Assume that $V$ is locally, at $P$, 
a complete intersection. Then
\begin{enumerate}
\item $\Omr$ is torsion free if and only if $V$ is non-singular in codimension 1 at $P$.
\item $\Omr$ is reflexive if and only if $V$ is non-singular in codimension 2 at $P$.
\end{enumerate}
\end{thm}

In the statement of the theorem, \textit{non-singular in codimension i at P} means that $\codim(R/\p)\geq i+1$, for all $\p\in\Sing(R)$, 
where $\codim(R/\p)=\dim R-\dim R/\p$ and $\Sing(R)=\{\p\in\Sp(R)|R_{\p}\mbox{ is not regular}\}$.

The first statement of the theorem was generalized to the module of high order K\"ahler differentials of a hypersurface, following the strategy in \cite{S}.

\begin{thm}\cite[Theorem 4.3]{BD}\label{thm BD}
Let $R$ be the local ring of a point $P$ on an irreducible hypersurface $W$ over a perfect field $\K$. Then $\Omk$ is torsion free 
if and only if $W$ is normal at $P$.
\end{thm}

In the next section we recall the notion of $k$-torsion freeness of an $R$-module, for any positive integer $k$. If $R$ is Noetherian 
and reduced, then the notions of torsion freeness and reflexiveness correspond, respectively, to $1$-torsion freeness and $2$-torsion 
freeness. Our main goal in this paper is to generalize Theorem \ref{thm BD} to apply to $k$-torsion freeness, for any $k\geq1$.


\section{A general theorem on $k$-torsion freeness}\label{sect k tors}

In this section we recall the notion of $k$-torsion freeness of a module. Then we give a characterization of this notion for modules
having projective dimension less or equal than 1. 

Let $R$ be a Noetherian ring and let $M$ be an $R$-module. The dual of $M$, denoted by $M^*$, is the module $\Hom_R(M,R)$. The bidual of $M$ is denoted by $M^{**}$. The bilinear map $\phi:M\times M^*\rightarrow R$ defined by $\phi(m,\varphi)=\varphi(m)$ induces an $R$-homomorphism $f:M\rightarrow M^{**}$, given by $f(m)=\phi(m,\cdot)$. For a given $R$-homomorphism $\varphi:M\rightarrow N$, we denote as $\varphi^*$ the induced map $N^*\rightarrow M^*$.  

Let us suppose that  $M$ is a finite $R$-module, i.e., $M$ is finitely generated. Since $R$ is Noetherian, $M$ is finitely presented, i.e., 
there exists an exact sequence 
$$P_1\stackrel{\varphi}\rightarrow P_0\rightarrow M\rightarrow \0,$$
where $P_0, P_1$ are finite free $R$-modules. Let $D(M):=\Coker(\varphi^*)$, which is known as the Auslander transpose of $M$. In \cite{Aus} it is shown that the previous sequence induces the following exact sequence:
\begin{equation}\label{fund seq}
\0\rightarrow \Ext^1_R(D(M),R)\rightarrow M\stackrel{f}\rightarrow M^{**}\rightarrow \Ext^2_R(D(M),R)\rightarrow \0.
\end{equation}
It is proved in \cite{AB} that for any $i\in\N$, $\Ext^i_R(D(M),R)$ depends only on $M$ and not on the particular presentation $P_1\rightarrow P_0\rightarrow M\rightarrow  \0$, where $P_0$ and $P_1$ are projective $R$-modules.

\begin{rem}
Recall that an $R-$module $M$ is \textit{torsionless} if $f$ is injective and that $M$ is \textit{reflexive} if $f$ is an isomorphism.  
Let $Q$ be the total quotient ring of $R$. Then $M$ is called {\it torsion free} if the natural map $\theta:M\rightarrow M_{Q}$ is injective, where $M_Q:=M\otimes_R Q$. It is known that $\ker(\theta)\subset\ker(f)$. Thus, the concept of torsionless implies the concept of torsion free. If $R$ is Gorenstein and has no embedded primes then the concepts are equivalent (see \cite[Theorem (A.1)]{Va}).
\end{rem}

In view of (\ref{fund seq}), torsionless and reflexiveness are respectively equivalent to $\Ext^1_R(D(M),R)=\0$ and $\Ext^1_R(D(M),R)=\Ext^2_R(D(M),R)=\0$. This leads us to the following general notion of $k$-torsion freeness.

\begin{defn}\cite{AB}
Let $k\in\N$, $k\geq1$. We say that the $R$-module $M$ is {\it $k$-torsion free} if $\Ext^i_R(D(M),R)=\0$, for $i\in\{1,\dots,k\}$.
\end{defn}

We want to study the $k$-torsion freeness of modules having projective dimension less or equal than one. For that, we need to
recall some results concerning the grade and depth of modules.

Let $R$ be a Noetherian ring, $M$ be a finite $R$-module and $I$ be an ideal of $R$. Recall that the \textit{grade of the ideal} $I$ \textit{ove}r $M$, denoted as $\grade(I,M)$, is the maximal size of a $M$-regular sequence in $I$. It is known 
that $\grade(I,M)$ can be computed in the following way (see \cite[Theorem 1.2.5]{BH}):
$$\grade (I,M)=\min\{i\in \N:\Ext_R^i(R/I, M)\neq\0\}.$$
We also define $\grade(M):=\min\{i\in\N:\Ext^i_R(M,R)\neq\0\}$. In addition, for a local ring $(R,\m)$ we denote $\depth(M):=\grade(\m,M)$. Then, by \cite[Proposition 1.2.10]{BH} we have
\begin{equation}\label{grade-depth}
\grade(M)=\grade(\Ann(M),R)=\min\{\depth(R_{\p}):\p \in \Supp(M)\}.
\end{equation}
We also need some facts regarding Cohen-Macaulay rings. If $(R,\m)$ is a finite-dimensional local Noetherian ring, then $\het(\p)\geq\depth(R)-\dim(R/\p)$. Moreover, if $R$ is Cohen-Macaulay, from this inequality we deduce that 
\begin{equation}\label{depth-codim}
	\depth(R_\p)=\dim R-\dim(R/\p)=\codim(R/\p).
\end{equation}

With these tools at hand, now we can give a characterization of $k$-torsion freeness for modules having projective dimension less or equal than one. This characterization is based on part of the proof of Lipman's theorem \ref{thm L}.

\begin{lem}\label{tor0-ext}
Let $M$ be an $R$-module and $Q$ be the total quotient ring of $R$. If $M_Q=\0$, then $\Ext_R^0(M,R)=\0$.
\end{lem}
\begin{proof}
Let $Z(R)$ be the set of zero divisors of $R$, i.e., $Z(R)=\cup_{P\in {\rm Ass}(R)}P$. As $M_Q=\0$, we have $M_P=\0$ for every $P\in {\rm Ass}(R)$, which implies ${\rm Ann}(M)\nsubseteq Z(R)$. Thus, there exists $x\in {\rm Ann}(M)$ such that $x\notin Z(R)$. Therefore $0<\grade(\Ann(M),R)=\grade(M)=\min\{i\in\N:\Ext^i_R(M,R)\neq\0\}$. We conclude $\Ext_R^0(M,R)=\0$.
\end{proof}

\begin{thm}\label{k torsion}
Let $R$ be a Noetherian local ring with total quotient ring $Q$. Let $M$ be a finite $R$-module with a finite projective resolution 
$$\0\rightarrow P_1\stackrel{\varphi}\rightarrow P_0\rightarrow M\rightarrow \0.$$ 
Let $k$ be a positive integer. Then $M$ is $k$-torsion free if and only if $\depth(R_{\p})\geq k+1$, for any $\p\in \Supp(D(M))$. Moreover, if $R$ is Cohen-Macaulay, $M$ is $k$-torsion free if and only if $\codim(R/\p)\geq k+1$, for any $\p\in\Supp(D(M))$.
\end{thm}
\begin{proof}
Using the projective resolution of $M$ and the definition of $\Ext_R^1(M,R)$, it follows that $\Ext_R^1(M,R)=D(M)$. As the functor $\Ext_R^1(\cdot,\cdot)$ commutes with localization, we obtain
\begin{align}
D(M)\otimes_R Q&=\Ext_R^1(M,R)\otimes_R Q\notag\\
&\cong \Ext^1_{R\otimes_R Q}(M\otimes_R Q, R\otimes_R Q)\notag\\
&=\Ext_Q^1(M_Q,Q).\notag
\end{align}
Since $Q$ is the total quotient ring of $R$, any non-unit of $Q$ is a zero divisor of $Q$, so $\depth(Q)=0$. Moreover $\dimproj(M_Q)\leq \dimproj(M)\leq 1$, the last inequality is by hypothesis. Using the Auslander-Buchsbaum formula 
$$0=\depth(Q)=\dimproj(M_Q)+\depth(M_Q),$$ 
we conclude $\dimproj(M_Q)=0$, so $M_Q$ is projective. It follows that $\0=\Ext_Q^1(M_Q,R)\cong D(M)\otimes_R Q=D(M)_{Q}$. By lemma \ref{tor0-ext}, $\Ext^{0}_R(D(M),R)=\0$. It follows that $M$ is $k$-torsion free if and only if $\Ext^i_R(D(M),R)=\0$ for every $i\in\{0,\dots, k\}$. 

On the other hand, $\Ext^i_R(D(M),R)=\0$ for every $i\in\{0,\dots, k\}$ if and only if $\grade(D(M))\geq k+1$. By (\ref{grade-depth}), $\grade(D(M))\geq k+1$ if and only if $\depth(R_p)\geq k+1$ for every $\p\in\Supp(D(M))$. If, in addition, $R$ is Cohen-Macaulay, by (\ref{depth-codim}), $\depth(R_{\p})\geq k+1$ if and only if $\codim(R/\p)\geq k+1$ for every $\p\in\Supp(D(M))$.
\end{proof}


\section{A characterization of $k$-torsion freeness}

Now we are ready to generalize Theorem \ref{thm BD} for any $k\geq1$. Throughout this section we use the 
following notation:

\begin{itemize}
\item $\K$ is a perfect field.
\item $A=\Pol/\langle f \rangle$, where $f\in\Pol$ is irreducible.
\item $W=\Sp(A)$.
\item $R$ is the local ring of a closed point $P\in W$.
\end{itemize}

Our first goal is to describe the support of the module $\Ext^1_{R}(\Omk,R)$ in terms of the singular locus of $R$. First we
need the following criterion of regularity for hypersurfaces in terms of the module of differentials of high order. 

\begin{prop}\label{crit reg}
Let $\p\in W$. Then $\Ap$ is a regular ring if and only if $\Oma$ is a free $\Ap$-module. In addition, in this case the rank of $\Oma$ is $L-1$, where 
$L=\binom{s-1+n}{s-1}$.
\end{prop}
\begin{proof}
The ``only if" part is well-known and holds in full generality (see, for instance, \cite[Section 4.2]{LT}). We include the proof here for the sake of completeness.

If $\Ap$ is a regular ring then $\Omega^1_{\Ap/\K}$ is free of rank $s-1$. In addition, in this case,
$\textbf{S}^{n}(I_{\Ap}/I_{\Ap}^{2})=I_{\Ap}^{n}/I_{\Ap}^{n+1}$, where $\textbf{S}^{n}(\cdot)$ denotes the nth-symmetric product. 
It follows that $I_{\Ap}^{n}/I_{\Ap}^{n+1}$ is free of rank $\binom{s-1+n-1}{s-2}$. Using the exact sequences
$$0\rightarrow I_{\Ap}^{n}/I_{\Ap}^{n+1}\rightarrow I_{\Ap}/I_{\Ap}^{n+1}\rightarrow I_{\Ap}/I_{\Ap}^{n}\rightarrow0,$$
it follows by induction that $\Oma$ is free of rank $L-1$.

Now assume that $\Oma$ is a free $\Ap$-module. We first show that the rank of this module is $L-1$. Let $\Omw$ be the sheaf of K\"ahler differentials 
of order $n$ of $W$. By the assumption, there exists an open subset $U\subset W$ such that $\p\in U$ and $\Omw|_U$ is free. In particular, $(\Omw|_U)_{\q}\cong\Omega^{(n)}_{A_{\q}/\K}$ 
is a free $A_{\q}$-module for all $\q\in U$. Since $W$ is irreducible, $U$ is irreducible as well, and so the rank of $\Omega^{(n)}_{A_{\q}/\K}$ is constant in $U$. 
Let $\q'\in U\subset W$ be a non-singular point (it exists because $W$ is irreducible and so $U$ and the open subset of non-singular points 
of $W$ are dense). By the ``only if" part of the proposition, $\Omega^{(n)}_{A_{\q'}/\K}$ is free of rank $L-1$. It follows that the rank of $\Oma$ is also $L-1$.

Now we show that $\Oma$ free implies that $\Ap$ is regular.
We can assume that $U=D(g)\cong\Sp(A_g)$ is a principal open set. Since $A_g$ is commutative with unit, there
exists $\m\subset A_g$ a maximal ideal such that $\p\subset\m$ \cite[Corollary 1.4]{AM}. Then 
$\Ap \cong (A_g)_\p \cong ((A_g)_{\m})_{\p} \cong (A_{\m})_{\p}.$ Since $\m\in U$, it follows that $\Omega^{(n)}_{A_{\m}/\K}$ 
is free of rank $L-1$. By \cite[Theorem 3.1]{BD}, $\m$ being a maximal ideal implies that $A_\m$ is a regular ring . We conclude that $\Ap\cong(A_{\m})_\p$ is also a regular ring.
\end{proof}

\begin{rem}
It was proved in \cite[Theorem 10.2]{BJNB} that the previous criterion of regularity holds more generally for local domains $(R,\m,\mathbb{K})$ 
with pseudocoefficient field $K$ such that Frac$(R)$ is separable over $K$ and $\mathbb{K}$ is perfect. In particular, it holds for arbitrary irreducible varieties. 
In addition, an algebraic proof of the second part of the proposition can also be deduced from \cite[Proposition 2.20]{BJNB}.
\end{rem}

\begin{lem}\label{ext1 free}
Let $S$ be a Noetherian local ring and $M$ a finite $S$-module such that $\dimproj(M)\leq 1$. Then $\Ext_S^1(M,S)=\0$ if and only if 
$M$ is free.
\end{lem}
\begin{proof}
Suppose that $\Ext_S^1(M,S)=\0$, so $\Ext_S^1(M,F)=\0$ for any $S$-free module $F$. As $\dimproj(M)\leq 1$ there exists an exact sequence
\begin{equation}\label{seq pd1}
\0\rightarrow F_1\stackrel{\varphi}\rightarrow F_0\rightarrow M\rightarrow\0,
\end{equation}
where $F_0$ and $F_1$ are finite free $S$-modules. Thus,
\begin{equation}
\0=\Ext_S^1(M,F_1)=\Coker\left(\begin{array}{clc}\Hom(F_0,F_1)&\rightarrow& \Hom(F_1,F_1)\\f&\mapsto&f\varphi\end{array}\right).
\end{equation}
Therefore, there exists $f\in \Hom(F_0,F_1)$ such that $f\varphi={\rm id}_{F_1}$; then $\varphi$ splits and $F_0\cong M\oplus F_1$. Thus $M$ is projective and since $S$ is a Noetherian local ring we conclude that $M$ is free.

The converse of this lemma is immediate, because $M$ is projective if and only if $\Ext_S^1(M,N)=\0$ for every $S$-module $N$.
\end{proof}

The next corollary follows the line of the proof of  \cite[Proposition 5.2]{L}. The crucial additions are proposition \ref{crit reg} and the 
fact that $\Omk$ has projective dimension less or equal than 1.

\begin{cor}\label{sing}
With the established notation, $$\Supp(\Ext^1_{R}(\Omk,R))=\Sing(R).$$
\end{cor}
\begin{proof}
Let $\p\in\Sp(R)$ be such that $R_{\p}$ is regular. Then $\Omega^1_{R_{\p}/\K}$ is a free $R_{\p}$-module and
so the same is true for $\Omega^{(n)}_{R_{\p}/\K}$. Since the module of differentials of high order commutes with localization 
(\cite[Theorem II-9]{N}), lemma \ref{ext1 free} implies
$$\0=\Ext^1_{R_{\p}}(\Omega^{(n)}_{R_{\p}/\K},R_{\p})=\big(\Ext^1_{R}(\Omk,R)\big)_{\p}.$$
This shows that $\Supp(\Ext^1_{R}(\Omk,R))\subset\Sing(R)$. 

Now let $\p\in\Sp(R)$ be such that $(\Ext^1_{R}(\Omk,R))_{\p}=\0$. This implies 
$\Ext^1_{R_{\p}}(\Omega^{(n)}_{R_{\p}/\K},R_{\p})=\0$. By \cite[Theorem 4.3]{BD}, $\dimproj(\Omk)\leq1$. Thus, by lemma
\ref{ext1 free}, $\Omega^{(n)}_{R_{\p}/\K}$ is a free $R_{\p}$-module. On the other hand, by the correspondence of prime ideals
in $R$ and $A$, we have $R_{\p}\cong \Ap$. In particular, $\Omega^{(n)}_{\Ap/\K}$ is a free $\Ap$-module. By proposition \ref{crit reg},
$\Ap$ is a regular ring. Thus $R_{\p}$ is regular and so $\Sing(R)\subset\Supp(\Ext^1_{R}(\Omk,R))$. 
\end{proof}

\begin{thm}\label{thm kBD}
Let $k\geq1$. Then $\Omk$ is $k$-torsion free if and only if $W$ is non-singular in codimension $k$ at $P$.
\end{thm}
\begin{proof}
As before, $\dimproj(\Omk)\leq1$. Consider the following projective resolution of $\Omk$:
$$\0\rightarrow F_1\stackrel{\varphi}\rightarrow F_0\rightarrow \Omk\rightarrow\0.$$
Since $R$ is Cohen-Macaulay, we can apply theorem \ref{k torsion} to obtain that $\Omk$ is $k$-torsion free if and only if
$\codim(R/\p)\geq k+1$ for any $\p\in\Supp(\Coker(\varphi^*))$. In addition, using the previous exact sequence we obtain
$\Ext^1_{R}(\Omk,R)=\Coker(\varphi^*)$. By corollary \ref{sing} we conclude that $\Omk$ is $k$-torsion free if and only 
if $\codim(R/\p)\geq k+1$ for any $\p\in\Sing(R)$.
\end{proof}

\begin{rem}
Notice that the entire strategy to prove theorem \ref{thm kBD} can also be used to generalize Lipman's theorem \ref{thm L} for 
$k$-torsion, for any $k\geq1$. 
\end{rem}

\begin{rem}
One of the key ingredients of the proof of Theorem \ref{thm kBD} was the fact that $\dimproj(\Omk)\leq1$, where $R$ is a local ring 
of an irreducible hypersurface. If this fact were also true for reduced complete intersections, then exactly the same strategy would give the analogous
statement of Theorem \ref{thm kBD} in this case. In this regard, an explicit presentation of $\Omn$ was recently given in 
\cite[Theorem 2.8]{BD} for any finitely generated $K$-algebra $R$. Using this presentation one could try to compute the projective 
dimension of $\Omn$, at least in some examples of complete intersections. Unfortunately, due to the size of the matrix giving the 
presentation, we did not succeed in computing any example for $n>1$, even with the help of a (modest) computer.
\end{rem}

\begin{rem}
Even though the main goal of this paper was to generalize theorem \ref{thm BD}, the results presented in section \ref{sect k tors} apply
to more general modules satisfying, among other hypothesis, that their projective dimension is less or equal than one. Families of modules
satisfying this hypothesis can be constructed as in \cite[Remark 2.1]{V}, \cite[Lemma 1]{L2}, or \cite[Proposition 1.6]{OZ}.
\end{rem}

\section*{Acknowledgements}

We want to thank the referee for her/his comments that improved the presentation of the paper and for the suggestions to simplify some of our proofs.

\end{document}